\documentclass[12pt]{amsart}
\usepackage{lipsum}
\usepackage{amssymb}
\usepackage{enumerate}
\usepackage{amsthm}
\usepackage{tikz-cd}
\usepackage{amsmath}
\usepackage[mathscr]{euscript}
\usepackage[utf8]{inputenc}
\usepackage[english]{babel}
\usepackage{hyperref}
\hypersetup{
    colorlinks=true,
    linkcolor=blue,
    citecolor=red,
    filecolor=red,      
    urlcolor=violet,
}

\makeatletter
\@namedef{subjclassname@2020}{%
  \textup{2020} Mathematics Subject Classification}
\makeatother

\newtheorem{thm}{Theorem}[section]
\newtheorem{lem}[thm]{Lemma}

\newtheorem{corl}[thm]{Corollary}
\newtheorem{xrem}{\rm Remark}
\newtheorem{exm}[thm]{\rm Example}

\DeclareMathOperator{\Nef}{{Nef}}

\DeclareMathOperator{\rank}{{rank}}
\DeclareMathOperator{\NE}{{NE}}

\DeclareMathOperator{\Div}{{Div}}
\DeclareMathOperator{\Bl}{{Bl}}

\begin{document}
\baselineskip=17pt

\subjclass[2020]{Primary 14C20, 14E30, 14E25, 14J60 ; Secondary 14J26}
\keywords{Nef cone, Mori cone of curves, Seshadri constants, Projective bundle, semistability}

\author{Snehajit Misra}
\author{Nabanita Ray} 

\address{Tata Institute of Fundamental Research (TIFR), Homi Bhabha Road, Mumbai 400005, India.}
\email[Snehajit Misra]{smisra@math.tifr.res.in}

\address{Tata Institute of Fundamental Research (TIFR), Homi Bhabha Road, Mumbai 400005, India.}
\email[Nabanita Ray]{nray@math.tifr.res.in}

\begin{abstract}
In this article, we give a description of the closed cone of curves of the projective bundle $\mathbb{P}(E)$ over a smooth 
projective variety $X$. Using duality, we then calculate the nef cone of divisors in $\mathbb{P}(E)$ over some special surfaces $X$ and for some special bundles on $X$.
As an application, we also calculate the Seshadri constants of  semistable ample vector bundles with vanishing discriminant on some special ruled surfaces at special points.
\end{abstract}

\title{Nef cones of projective bundles over surfaces and Seshadri constants}
\maketitle

\section{Introduction}
 A fundamental invariant of a projective variety $X$ is its cone of nef divisors, denoted by $\Nef(X)$. It encodes all the information about embeddings of $X$ in projective spaces. Also, knowledge of these cones can be used to study positivity questions, interpolations problems, Seshadri constants etc. Nef cones of a wide class of varieties have been calculated over the last few years (see \cite{L1}, \cite{Fu}, \cite{M-O-H}, \cite{K-M-R}).

In his paper \cite{M}, Yoichi Miyaoka studied the nef cone of a projective bundle $\mathbb{P}(E)$ over a smooth 
complex projective curve $C$, where $E$ is any semistable vector bundle of rank $r$ on $C$. More generally, it is showed in \cite{Fu} that $\Nef\bigl(\mathbb{P}(E)\bigr)$ is determined by the numerical data $\mu_{\min}(E)$ in the Harder-Narasimhan filtration of $E$. In these cases, the Picard number of $\mathbb{P}(E)$ is 2, and hence the nef cones are generated by two extremal rays in
a two dimensional space. In general, when the Picard number is at least 3, the nef cones might not be a finite polyhedron, and hence are not so easy to calculate. For example if $E$ is a rank 2 bundle obtained by the
Serre construction from the ideal sheaf of 10 very general points on $\mathbb{P}^2$, then the positivity
of $E$ is related to the Nagata conjecture. Thus one has to settle for special bundles $E$ on
special varieties $X$, even when dimension of $X$ is 2. 
Motivated by this, we consider a projective bundle $\mathbb{P}(E)$ over a smooth irreducible complex projective variety $X$ together with the projection 
$\pi : \mathbb{P}(E) \longrightarrow X$. In section 3, we prove the following:
\begin{thm}\label{thm1.1}
 Let $E$ be a  vector bundle of rank $r \geq 2$ on a smooth complex projective variety $X$. For an irreducible curve $C$ in $X$ together with its normalization $\eta_c : \tilde{C}\longrightarrow C$, consider the following fibre product diagram: 
\begin{center}
 \begin{tikzcd}
\tilde{C}\times_X\mathbb{P}(E)=\mathbb{P}\bigl(\eta_c^*(E\vert_C)\bigr) \arrow[r, "\tilde{\eta_c}"] \arrow[d, "\tilde{\pi_c}"] & C\times_X \mathbb{P}(E) = \mathbb{P}(E\vert_C) \arrow[r, "j"] \arrow[d,"\pi_c"] & \mathbb{P}(E) \arrow[d,"\pi"] \\
\tilde{C} \arrow[r, "\eta_c"]          & C \arrow[r, "i"]                                         & X
\end{tikzcd}
\end{center}
where $i$ and $j$ are inclusions. We define $\psi_c:= j \circ \tilde{\eta_c}$. 
Then the closed cone of curves in $\mathbb{P}(E)$ is given by
 \begin{center}
$ \overline{\NE}\bigl(\mathbb{P}(E)\bigr) = \overline{\sum\limits_{C\in \Gamma}\bigl(\psi_{c}\bigr)_* \Bigl(\overline{\NE}\bigl(\mathbb{P}(\eta_c^*(E\vert_C))\bigr)\Bigr)},$
  \end{center}
 where $\Gamma$ is the set of all irreducible curves in $X$.
\end{thm}
Consequently, we calculate the nef cones $\Nef\bigl(\mathbb{P}(E)\bigr)$ of some special projective bundles $\mathbb{P}(E)$ over some special smooth surfaces $X$ (see Corollary \ref{corl3.4}, Corollary \ref{cor3.5}) by applying duality to Theorem \ref{thm1.1}. An important finding in section 3 is the following:
\begin{corl}
Let $X$ be a smooth complex projective surface with
 $\overline{\NE}(X) = \bigl\{ a_1[C_1] + a_2[C_2] +\cdots+a_n[C_n] \mid a_i \in \mathbb{R}_{\geq 0}\bigr\}$ for  some irreducible curves $C_1,C_2,\cdots,C_n$  in $X$. If $E$ is a semistable vector bundle of rank $r\geq2$ on $X$ with vanishing
discriminant, i.e. $\triangle(E) := 2rc_2(E)-(r-1)c_1^2(E) = 0$, then 
\begin{center}
  $ \Nef\bigl(\mathbb{P}(E)\bigr)
   =\Bigl\{ y_0\xi+\pi^*\gamma \mid \gamma\in N^1(X)_{\mathbb{R}},  y_0\geq 0$, $   y_0\mu(E\vert_{C_j}) + (\gamma\cdot C_j)\geq0$ for all $ 1\leq j\leq n \Bigr\}$,
\end{center}
where $\xi$  denotes the numerical equivalence classes of 
the tautological bundle $\mathcal{O}_{\mathbb{P}(E)}(1)$.
 \end{corl}

We also give several examples to illustrate our results. In most of these examples, the projective bundles  $\mathbb{P}(E)$ have Picard number at least 3, and the nef cones $\Nef\bigl(\mathbb{P}(E)\bigr)$ are all finite polyhedra.
\vspace{1mm}

A vector bundle $V$ on $X$ is called ample if $\mathcal{O}_{\mathbb{P}(V)}(1)$ is ample on $\mathbb{P}(V)$. The Seshadri constant $\varepsilon(V,x)$ of an ample vector bundle $V$ on a projective variety X at a closed point $x \in X$ is introduced in \cite{Hac}.  Let us consider the following pullback diagram under the blow up map $\rho_x : \tilde{X}_x = \Bl_x(X) \longrightarrow X$
\begin{center}
\begin{tikzcd} 
\mathbb{P}(V)\times_X \tilde{X}_x = \mathbb{P}\bigl(\rho_x^*(V)\bigr) \arrow[r, "\tilde{\rho_x}"] \arrow[d, "\pi"]
 & \mathbb{P}(V) \arrow[d,"\pi'"]\\
 \tilde{X}_x \arrow[r, "\rho_x" ]
 & X
\end{tikzcd}
\end{center}
 Let $\tilde{\xi}_x$ be the numerical equivalence class of the tautological bundle $\mathcal{O}_{\mathbb{P}(\rho^*_xV)}(1)$, and 
 $\tilde{E}_x := \tilde{\rho_x}^{-1}(F_x)$, where $F_x$ is the class of the fibre of the map $\pi'$ over the point $x$. Then the 
 Seshadri constant of $V$ at $x\in X$ is defined as 
 \begin{align*}
  \varepsilon(V,x) := \sup\Bigl\{ \lambda \in \mathbb{R}_{>0} \mid \tilde{\xi}_x - \lambda\tilde{E}_x \rm\hspace{1mm} is \hspace{1mm} nef \Bigr\}. 
 \end{align*}
Note that the Seshadri constant $\varepsilon(V,x)$ is not equal to the Seshadri constant of the tautological line bundle $\mathcal{O}_{\mathbb{P}(V)}(1)$ on $\mathbb{P}(V)$. These interesting constants have been extensively studied by many authors especially for line bundles on algebraic surfaces (see \cite{L1}, Chapter 5 and \cite{B-D-H-K-K-S-S}). In \cite{Hac}, it is shown that if $V$ is an ample vector bundle on a smooth projective curve $C$, then the Seshadri constant $\varepsilon(V,x) = \mu_{\min}(V)$ at $x\in C$. Recently, Seshadri constants for equivariant bundles have been calculated on toric varieties in \cite{HMP} and on certain homogeneous varieties in \cite{BHN20}.

As an application of our results in section 3, we calculate the Seshadri constant of semistable ample vector bundles with vanishing discriminant on some special ruled surfaces at special points (see Theorem \ref{thm4.2}).
\section{Preliminaries}
Throughtout this article, all the varieties are assumed to be over the field of complex numbers $\mathbb{C}$.
 In this section, we recall the definition and duality property of nef cone as well as the definition of semistability of vector bundles on smooth projective varieties. Let $X$ be a smooth complex projective variety of dimension $n$ with a fixed ample line bundle $H$ on it.
 \subsection{Harder-Narasimhan Filtration:}
 A non-zero torsion-free coherent sheaf $\mathcal{G}$ on $X$ is said to be $H$-semistable if $$\mu_H(\mathcal{F}) = \frac{c_1(\mathcal{F})\cdot H^{n-1}}{\rank(\mathcal{F})} \leq \mu_H(\mathcal{G}) = \frac{c_1(\mathcal{G})\cdot H^{n-1}}{\rank(\mathcal{G})}$$ for 
all subsheaves $\mathcal{F}$ of $\mathcal{G}$.
For every vector bundle $E$ on $X$, there is a unique filtration
\begin{align*}
 0 = E_0 \subsetneq E_1 \subsetneq E_2 \subsetneq\cdots\subsetneq E_{k-1} \subsetneq E_k = E
\end{align*}
of subbundles of $E$, called the Harder-Narasimhan filtration of $E$, such that $E_i/E_{i-1}$ is $H$-semistable torsion free sheaf
for each $i \in \{ 1,2,\cdots,k\}$
and $\mu_H\bigl(E_i/E_{i-1}\bigr) > \mu_H\bigl(E_{i+1}/E_{i}\bigr)$ for each $i \in \{1,2,\cdots,k-1\}$. 
We define $Q_k := E_{k}/E_{k-1}$ and $\mu_{\min}(E) := \mu_H(Q_k) = \mu_H\bigl(E_{k}/E_{k-1}\bigr)$. 
\subsection{Nef cone} Let 
  $\Div^0(X) := \bigl\{ D \in \Div(X) \mid D\cdot C = 0 $ for all curves $C$ in $X \bigr\}$ be the subgroup of $\Div(X)$ consisting of numerically trivial divisors. The quotient $\Div(X)/\Div^0(X)$ is called the N\'{e}ron Severi group of $X$, and is denoted by $N^1(X)_{\mathbb{Z}}$.
   The N\'{e}ron Severi group  $N^1(X)_{\mathbb{Z}}$ is a free abelian group of finite rank.
 Its rank, denoted by $\rho(X)$ is called the Picard number of $X$. In particular,  $N^1(X)_{\mathbb{R}}  := N^1(X)_{\mathbb{Z}} \otimes \mathbb{R} = \bigl(\Div(X)/\Div^0(X)\bigr) \otimes \mathbb{R}$ is called the real N\'{e}ron Severi group.
 
 A Cartier divisor $D$ on $X$
 (with $\mathbb{Z}, \mathbb{Q}$ or $\mathbb{R}$ coefficients) is called nef if $D\cdot C \geq 0$ for all irreducible curves $C \subseteq X$. The intersection product being independent of numerical equivalence class, one can talk about nef classes in $N^1(X)_\mathbb{R}$. The convex cone of all nef classes in $N^1(X)_\mathbb{R}$ is called the nef cone of $X$, and is denoted by $\Nef(X)$.
 
 A formal sum $\gamma = \sum\limits_{i}a_iC_i$, where $a_i \in \mathbb{R}$ and $C_i \subseteq X$ are irreducible curves in $X$, is called a real 1-cycle in $X$. The $\mathbb{R}$-vector space of all real 1-cycle in $X$ is denoted by $Z_1(X)_\mathbb{R}$. Two elements $\gamma_1,\gamma_2 \in Z_1(X)_\mathbb{R}$ are called numerically equivalent, denoted by $\gamma_1 \equiv \gamma_2$ if $D\cdot\gamma_1 = D\cdot \gamma_2$ for every divisor $D$ in $X$. We denote $N_1(X)_\mathbb{R} = Z_1(X)_\mathbb{R}/\equiv$. The cone of curves in $X$, denoted by $\NE(X)$ is the cone spanned by the effective 1-cycle classes in $N_1(X)_{\mathbb{R}}$. If $C$ is an effective curve, then we denote the numerical equivalence class of curves by $[C]$.
 
 The intersection pairing $ N^1(X)_\mathbb{R} \times N_1(X)_\mathbb{R} \longrightarrow \mathbb{R}$ induces a perfect pairing. The closed cone of curves $\overline{\NE}(X)$, which is defined to be the closure of $\NE(X)$, is known to be the dual to the nef cone $\Nef(X)$, i.e.
 \begin{center}
 $\overline{\NE}(X) = \bigl\{ \gamma \in N_1(X)_{\mathbb{R}} \mid \alpha \cdot \gamma \geq 0 $ for all $\alpha \in \Nef(X)\bigr\}.$
\end{center}
\section{Main Results}
We quickly recall our set up. Let $E$ be a vector bundle of rank $\geq 2$ on a smooth complex projective variety $X$, and $\pi:\mathbb{P}(E)\longrightarrow X$ be the projection.
In this section, we give a description of the closed cone $\overline{\NE}\bigl(\mathbb{P}(E)\bigr)$ of curves in $\mathbb{P}(E)$  in terms of the closed cone of curves $\overline{\NE}\bigl(\mathbb{P}(E\vert_C)\bigr)$ for every irreducible curve $C$ in $X$. 
\begin{proof}[Proof of Theorem \ref{thm1.1}]\label{prop3.3}
For an irreducible curve $C$ in $X$ together with its normalization $\eta_c : \tilde{C}\longrightarrow C$, we recall the following fibre product diagram: 
\begin{center}
 \begin{tikzcd}
\tilde{C}\times_X\mathbb{P}(E)=\mathbb{P}\bigl(\eta_c^*(E\vert_C)\bigr) \arrow[r, "\tilde{\eta_c}"] \arrow[d, "\tilde{\pi_c}"] & C\times_X \mathbb{P}(E) = \mathbb{P}(E\vert_C) \arrow[r, "j"] \arrow[d,"\pi_c"] & \mathbb{P}(E) \arrow[d,"\pi"] \\
\tilde{C} \arrow[r, "\eta_c"]          & C \arrow[r, "i"]                                         & X
\end{tikzcd}
\end{center}
where $i$ and $j$ are inclusions and $\psi_c:= j \circ \tilde{\eta_c}$. 
 Let $\xi \in N^1\bigl(\mathbb{P}(E)\bigr)_{\mathbb{R}}$ be the numerical equivalence class of the tautological line bundle $\mathcal{O}_{\mathbb{P}(E)}(1)$ on $\mathbb{P}(E)$. 
 Then $$N^1\bigl(\mathbb{P}(E)\bigr)_{\mathbb{R}} = \mathbb{R}\xi \oplus \pi^*N^1(X)_{\mathbb{R}}.$$
 We fix the notations $\xi_{\tilde{c}}$ and $f_{\tilde{c}}$ for the numerical equivalence classes of the tautological line bundle 
 $\mathcal{O}_{\mathbb{P}(\eta_c^*(E\vert_C))}(1)$ of $\mathbb{P}\bigl(\eta_c^*(E\vert_C)\bigr)$ and a fibre of the map 
 $\tilde{\pi_{c}}$ respectively. Then $N^1\bigl(\mathbb{P}(\eta_c^*(E\vert_C)\bigr)_{\mathbb{R}}$ is generated by $\xi_{\tilde{c}}$ and $f_{\tilde{c}}$.
 The map $\psi_c$ induces the map 
 \begin{center}
 $\psi_{c}^{*}:N^1\bigl(\mathbb{P}(E)\bigr)_{\mathbb{R}}\longrightarrow N^1\bigl(\mathbb{P}(\eta_c^*(E\vert_C)\bigr)_{\mathbb{R}}$ 
 \end{center}
 between the N\'{e}ron Severi groups such that
 \begin{center}
 $\psi_{c}^{*}(\xi)=\xi_{\tilde{c}}$ and $\psi_{c}^{*}\bigl(\pi^*L\bigr)= (L\cdot C)f_{\tilde{c}},$ for any $L\in N^1(X)_{\mathbb{R}}.$ 
 \end{center}
 As a consequence, we get the pushforward map 
 $$(\psi_c)_* : N_1\bigl(\mathbb{P}(\eta_c^*(E\vert_C)\bigr)_{\mathbb{R}} \longrightarrow N_1\bigl(\mathbb{P}(E)\bigr)_{\mathbb{R}}.$$
 Hence
 \begin{align*}
 \overline{\sum\limits_{C\in \Gamma}(\psi_{c})_{*} \Bigl(\overline{\NE}\bigl(\mathbb{P}(\eta_c^*(E\vert_C))\bigr)\Bigr)} \subseteq \overline{\NE}\bigl(\mathbb{P}(E)\bigr),
\end{align*}
where the sum is taken over the set $\Gamma$ of all irreducible curves in $X$.
Now to prove the reverse inequality, we consider the  numerical equivalence class $\bigl[\bar{C}\bigr]\in \NE\bigl(\mathbb{P}(E)\bigr)$ of an irreducible curve $\bar{C}$ in $\mathbb{P}(E)$ which is not contained in any fibre of $\pi$.
Denote $\pi(\bar{C})=C$. Then, $\bar{C}\subseteq \mathbb{P}(E\vert_C)$. 
Then there exists a unique irreducible curve $C'\subseteq \mathbb{P}\bigl(\eta_c^*(E\vert_C)\bigr)$  such that $\tilde{\eta_c}(C') = \bar{C}$ and $\bigl(\psi_{c}\bigr)_*\bigl([C']\bigr)=\bigl[\bar{C}\bigr]$. Also, the numerical equivalence classes of curves in a fibre of $\tilde{\pi_{c}}$ maps to the numerical classes of curves in a fibre of $\pi$ by 
 $\bigl(\psi_{c}\bigr)_*$.
Hence, we have
\begin{align*}
 \overline{\NE}\bigl(\mathbb{P}(E)\bigr)\subseteq\overline{\sum\limits_{C\in \Gamma}\bigl(\psi_{c}\bigr)_* \Bigl(\overline{\NE}\bigl(\mathbb{P}(\eta_c^*(E\vert_C))\bigr)\Bigr)}.
\end{align*}
This completes the proof.
\end{proof}
\begin{corl}\label{corl3.1}
Let $X$ be a smooth complex projective surface with
 $\overline{\NE}(X) = \bigl\{ a_1[C_1] + a_2[C_2] +\cdots+a_n[C_n] \mid a_i \in \mathbb{R}_{\geq 0}\bigr\}$ for  some irreducible curves $C_1,C_2,\cdots,C_n$  in $X$. If $E$ is a semistable vector bundle of rank $r\geq2$ on $X$ with vanishing
discriminant, i.e. $\triangle(E) := 2rc_2(E)-(r-1)c_1^2(E) = 0$, then 
\begin{center}
 $\Nef\bigl(\mathbb{P}(E)\bigr)
   =\Bigl\{ y_0\xi+\pi^*\gamma \mid \gamma\in N^1(X)_{\mathbb{R}},  y_0\geq 0,    y_0\mu(E\vert_{C_j}) + (\gamma\cdot C_j)\geq0 \text{ for all } 1\leq j\leq n\Bigr\}$,
\end{center}
where $\xi$  denotes the numerical equivalence class of 
the tautological bundle $\mathcal{O}_{\mathbb{P}(E)}(1)$.
\end{corl}
\begin{proof}
Let $C$ be an irreducible curve in $X$ such that $[C] = \sum\limits_{i=1}^n x_{i}[C_{i}] \in N_1(X)_{\mathbb{R}}$ for some $x_i \in \mathbb{R}_{\geq 0}$. As $E$ is semistable with vanishing discriminant, applying Theorem 1.2 in \cite{B-B} to the map
$i \circ \eta_c : \tilde{C} \longrightarrow  C  \hookrightarrow X$, we get that $\eta_c^*\bigl(E\vert_C\bigr)$ is also semistable bundle on $\tilde{C}$ for any irreducible curve $C$ in $X$. Using Theorem 3.1 in \cite{M}, we have
\begin{align*}
 \Nef\bigl(\mathbb{P}(\eta_c^*(E\vert_C))\bigr) = \Bigl\{ a\bigl(\xi_{\tilde{c}}-\mu(\eta_c^*(E\vert_C))f_{\tilde{c}}\bigr) + bf_{\tilde{c}} \mid a,b \in \mathbb{R}_{\geq 0}\Bigr\}.
\end{align*} 
We define $l_c:= \deg\bigl(\eta_c^*(E\vert_C)\bigr)-\mu\bigl(\eta_c^*(E\vert_C)\bigr).$ 
\vspace{2mm}

Applying duality, we then get
\begin{align*}
\overline{\NE}\bigl(\mathbb{P}(\eta_c^*(E\vert_C))\bigr) = \Bigl\{a(\xi_{\tilde{c}}^{r-2}f_{\tilde{c}}) + b(\xi_{\tilde{c}}^{r-1}-l_c\xi_{\tilde{c}}^{r-2}f_{\tilde{c}})\mid a,b \in\mathbb{R}_{\geq 0}\Bigr\}.
\end{align*}

Therefore, by Theorem \ref{thm1.1} we have
\vspace{2mm}

  $\overline{\NE}\bigl(\mathbb{P}(E)\bigr)=\overline{\sum\limits_{C\in \Gamma}(\psi_{c})_* \Bigl(\overline{\NE}\bigl(\mathbb{P}(\eta_c^*(E\vert_C))\bigr)\Bigr)} = \overline{\sum\limits_{C\in \Gamma}\Bigl[ \mathbb{R}_{\geq 0}(\xi^{r-2}F) + \mathbb{R}_{\geq 0}( \xi^{r-1}\pi^*[C] - l_c\xi^{r-2}F)\Bigr]}$,
  
  where $F$ is the numerical equivalence class of the fibre of the projection map $\pi : \mathbb{P}(E)\longrightarrow X$.
  \vspace{3mm}

Also,
$\mu\bigl(\eta_{c}^*(E\vert_C)\bigr) = \frac{\deg(E\vert_C)}{r}$
$= \frac{\sum\limits_{i=1}^nx_i\deg(E\vert_{C_i})}{r}$
$= \sum\limits_{i=1}^nx_i\mu\bigl(\eta_{c_i}^*(E\vert_{C_i})\bigr)$,
\vspace{1mm}

$l_{c_i} = \deg(\eta_{c_i}^*(E\vert_C)\bigr) - \mu\bigl(\eta_{c_i}^*(E\vert_{C_i})\bigr) = \frac{r-1}{r}\deg(E\vert_{C_i})$ for every $ 1\leq i \leq n$, and 
$l_c = \sum\limits_{i=1}^n x_il_{c_i}.$
\vspace{2mm}

This shows that $\overline{\NE}\bigl(\mathbb{P}(E)\bigr)$ is generated by
\begin{center}  
  $\Bigl\{ \xi^{r-2}F$ , $\xi^{r-1}\pi^*[C_i] - \frac{r-1}{r}\deg(E\vert_{C_i})\xi^{r-2}F \mid 1\leq i\leq n\Bigr\}$,
\end{center}

and the nef cone
\begin{align*}
 \Nef\bigl(\mathbb{P}(E)\bigr)
   =\Bigl\{ y_0\xi+\pi^*\gamma \mid \gamma\in N^1(X)_{\mathbb{R}},  y_0\geq 0,    y_0\mu(E\vert_{C_j}) + (\gamma\cdot C_j)\geq0 \text{ for all } 1\leq j\leq n\Bigr\}.
\end{align*}
\end{proof}
\begin{xrem}\label{xrem1}
\rm  Let $X$ be a smooth irregular complex projective surface $X$, i.e., $ q = H^1(X,\mathcal{O}_X) $\\
$\neq 0$, and the closed cone of curves $\overline{\NE}(X)$ is a finite polyhedron generated by classes of irreducible curves. Examples of such surfaces $X$ include ruled surface over a smooth curve of genus greater than 0 with closed cone of curves generated by irreducible curve classes, very general abelian surfaces with Picard number 1, del Pezzo surfaces etc. Then there exists non-split extension of the form 
 \begin{align*}
 0 \longrightarrow \mathcal{O}_C \longrightarrow E \longrightarrow \mathcal{O}_C \longrightarrow 0.
 \end{align*}
In this case, $E$ is a semistable bundle of rank 2 with vanishing discriminant. Moreover, for any positive integer $r$, the vector bundles of the forms $E^{\oplus r} \oplus \mathcal{O}_C$ and $E^{\oplus r}$ are examples of semistable bundles 
of ranks $2r+1$ and $2r$ respectively with vanishing discriminant. In this way, one can produce examples of semistable bundles with vanishing discriminant of any rank on such $X$. 

Also, for a semistable bundle $V$ on $X$ with discriminant 0, and a nonsplit extension of the form $0\longrightarrow V \longrightarrow E \longrightarrow V\longrightarrow 0$, $E$ is also semistable with vanishing discriminant. In all these cases, one can calculate the nef cones using Corollary \ref{corl3.1}.
\end{xrem}
\begin{exm}\label{exm3.2}
 \rm Let $ \rho : X = \mathbb{P}(W) \longrightarrow C$ be a ruled surface on a smooth complex projective curve $C$ defined by a normalized bundle (in the sense of Hartshorne [see \cite{Har1}, Ch 5]) $W$ with $\mu_{\min}(W) = \deg(W)$. Then $\overline{\NE}(X) = \{ a\zeta + bf \mid a,b\in \mathbb{R}_{\geq 0}\}$ is a finite polyhedron generated by classes of two irreducible curves, where $\zeta = \bigl[\mathcal{O}_X(\sigma)\bigr]$ (here $\sigma$ is the normalized section such that $\mathcal{O}_{\mathbb{P}(W)}(1) \cong \mathcal{O}_X(\sigma)$, and $f$ is the numerical equivalence class of a fibre of $\rho$).
 Let $V$ be a semistable (resp. stable) vector bundle of rank $r$ and degree $d$ on $C$. Then the pullback bundle  $E : = \rho^*(V)$ is also semistable
 (resp. stable) on $X$  with discriminant 
 $$\triangle(E) = \triangle(\rho^*V) =  2rc_2(\rho^*V)-(r-1)c_1^2(\rho^*V) = 0 -(r-1)df\cdot df = 0.$$
 Therefore, one can calculate the nef cone $\Nef\bigl(\mathbb{P}(E)\bigr)$ in this case using Corollary \ref{corl3.1}. Note that in this case $\mathbb{P}(E) \cong \mathbb{P}(W)\times_C \mathbb{P}(V)$, and nef cone of the fiber product $\Nef\bigl(\mathbb{P}(W)\times_C \mathbb{P}(V)\bigr)$ is already known due to \cite{K-M}. Our result in this case matches with the result in \cite{K-M}.
 
 For example, any indecomposable normalized rank 2 bundle $W$ over an elliptic curve $C$ is one of the following type (see Proposition 2.15, Chapter V, \cite{Har1}):
\begin{itemize}
\item
Either (i) $0\longrightarrow\mathcal{O}_C\longrightarrow W \longrightarrow \mathcal{O}_C \longrightarrow 0 $,

\item or \hspace{5mm} (ii) $0\longrightarrow\mathcal{O}_C\longrightarrow W \longrightarrow \mathcal{O}_C(p) \longrightarrow 0 $\hspace{5mm} 
for some closed point $p\in C$.
\end{itemize}

In both the cases, $\deg(W) \geq 0$, and hence $W$ is semistable.

We consider the ruled surface $\rho : X=\mathbb{P}(W) \longrightarrow C$ over the elliptic curve $C$ defined by the nonsplit extension 
$0\longrightarrow\mathcal{O}_C\longrightarrow W \longrightarrow \mathcal{O}_C \longrightarrow 0$.
Let us define $[C_1] := \bigl[\mathcal{O}_X(\sigma)\bigr] = \bigl[\mathcal{O}_{\mathbb{P}(W)}(1)\bigr]$, and $[C_2]:= f =$ numerical class of a fibre of the map $\rho$. Then $\overline{\NE}(X) = \{ a[C_1] + b[C_2] \mid a,b\in \mathbb{R}_{\geq 0}\}$.

Let $V$ be the bundle on $C$ which sits in the
nonsplit extension $0\longrightarrow\mathcal{O}_C\longrightarrow V \longrightarrow \mathcal{O}_C(p) \longrightarrow 0$ for some closed point $p$ in $C$. Then, $E:=\rho^*(V)$ is 
a stable bundle on $X$ with $\triangle(E) =0$. Note that, in this example, the Picard number of $\mathbb{P}(E)$ is 3.

Now, 
\begin{align*}
\deg(E\vert_{C_1}) = c_1(\rho^*V)\cdot \zeta = \deg(V)\zeta\cdot f = 1 ; \deg(E\vert_{C_2}) = c_1(\rho^*V)\cdot f = \deg(V)f\cdot f = 0.
\end{align*}
\begin{align*}
 b_{11} := C_1\cdot C_1 = \zeta^2 = \deg(W) = 0; \hspace{2mm} b_{12} = b_{21} := C_1\cdot C_2 = \zeta\cdot f = 1; \hspace{2mm} b_{22} := C_2\cdot C_2 = f^2 = 0.\end{align*} 
$y_0\mu(E\vert_{C_1}) + y_1b_{11} + y_2b_{21} = \frac{y_0}{2} +y_2$ ; $y_0\mu(E\vert_{C_2}) + y_1b_{12} + y_2b_{22} = y_1$.

Therefore, for the projective bundle $\pi : \mathbb{P}(E) \longrightarrow X$, the nef cone is 
\begin{center}
$\Nef\bigl(\mathbb{P}(E)\bigr) = \Bigl\{ y_0\xi + y_1(\pi^*\zeta) + y_2(\pi^*f) \mid y_0 \geq 0 , y_1 \geq 0, \frac{y_0}{2} + y_2 \geq 0\Bigr\}.$
\end{center}
\end{exm}
\begin{exm}\label{3.3}
\rm Let $\rho : X = \mathbb{P}(W) \longrightarrow C$ be a ruled surface over a smooth elliptic curve $C$ defined by the rank two bundle 
$W = \mathcal{O}_C \oplus \mathcal{O}_C$. Then $\overline{\NE}(X) = \{ a\zeta + bf \mid a,b\in \mathbb{R}_{\geq 0}\}$, where 
$\zeta = \bigl[\mathcal{O}_{\mathbb{P}(W)}(1)\bigr] \in N^1(X)$ and $f$ is the numerical equivalence class of a fibre of $\rho$.

Let $\rho_x : \tilde{X}_x  =  \Bl_x(X) \longrightarrow X$ be the blow up of $X$ at a closed point $x$ in a section $\sigma$ such that $\mathcal{O}_X(\sigma)\cong \mathcal{O}_{\mathbb{P}(W)}(1)$, and $E_x$ be the exceptional divisor. Then we claim that
$\overline{\NE}(\tilde{X}_x) =\bigl\{ a[C_1] + b[C_2] + c[C_3] \mid a,b,c \in \mathbb{R}_{\geq 0}\bigr\}$, where
$[C_1]=\rho^*_xf - E_x$, $[C_2]=\rho^*_x\zeta - E_x$ and $ [C_3]= E_x$. 

To prove our claim enough to show that if $V$ is an irreducible curve in $X$ which is not a fibre i.e. $[V]=a\zeta+b f$; $a> 0$ and $b\geq 0$, then $\text{mult}_{x}V=r\leq a$.

Let $F$ be a fibre passing through $x$. Then $\text{mult}_{x}F=1$. Hence
 $$a=V\cdot F=\sum_{P\in V\cap F}(V.F)_P\geq (V.F)_x\geq r \text{ (see \cite{Har1}, Chapter V, Proposition 1.4) }$$

Now we also have 

$b_{11} := C_1\cdot C_1 = (\rho^*_xf - E_x)\cdot(\rho^*_xf - E_x) = E_x^2 = -1$ ; 

$b_{22} := C_2\cdot C_2 =(\rho^*_x\zeta - E_x)\cdot(\rho^*_x\zeta - E_x) = \deg(W) - 1 = -1$ ;
$b_{33} := C_3\cdot C_3 = E_x^2 = -1$.

$b_{12} = b_{21} := C_1\cdot C_2=(\rho^*_xf - E_x)\cdot(\rho^*_x\zeta - E_x) = \zeta\cdot f+ E_x^2 = 0$ ; 

$b_{13} = b_{31} := C_1\cdot C_3 = (\rho^*_xf - E_x)\cdot E_x = 1$ ;
$b_{23} = b_{32} := C_2\cdot C_3 = (\rho^*_x\zeta - E_x)\cdot E_x=1$ ;

Let $E = \rho^*_x\bigl(\rho^*(V)\bigr)$, where $V$ is the semistable bundle on $C$ given by the nonsplit extension $0\longrightarrow\mathcal{O}_C\longrightarrow V \longrightarrow \mathcal{O}_C(p) \longrightarrow 0 $ 
for some closed point $p\in C$. Then $E$ is also semistable bundle with $\triangle(E) = 0$.
Note that, in this example, the Picard number of $\mathbb{P}(E)$ is 4.

$\deg(E\vert_{C_1}) = c_1\bigl(\rho^*_x\bigl(\rho^*(V)\bigr)\bigr)\cdot (\rho^*_xf - E_x) = 0$ ;
$\deg(E\vert_{C_2}) = c_1\bigl(\rho^*_x\bigl(\rho^*(V)\bigr)\bigr)\cdot (\rho^*_x\zeta - E_x) = 1$;

$\deg(E\vert_{C_3}) = c_1\bigl(\rho^*_x\bigl(\rho^*(V)\bigr)\bigr)\cdot  E_x = 0$.

$y_0\mu(E\vert_{C_1}) + y_1b_{11} + y_2b_{21} +y_3b_{31} = y_3 -y_1$ ;
$y_0\mu(E\vert_{C_2}) + y_1b_{12} + y_2b_{22} +y_3b_{32} = \frac{y_0}{2} + y_3 -y_2$ ;

$y_0\mu(E\vert_{C_3}) + y_1b_{13} + y_2b_{23} +y_3b_{33} = y_1 +y_2 -y_3$.
\vspace{1mm}

Therefore, the nef cone of the projective bundle $\pi : \mathbb{P}(E) \longrightarrow \tilde{X}_x$ is 
\begin{align*}
 \Nef\bigl(\mathbb{P}(E)\bigr) = \Bigl\{ y_0\xi+\sum\limits_{i=1}^3y_i\pi^*(C_i) \mid  y_0\geq 0, y_3 -y_1\geq 0, \frac{y_0}{2} + y_3 -y_2\geq 0,  y_1 +y_2 -y_3\geq 0\Bigr\}.
\end{align*}
\end{exm}
\begin{corl}\label{corl3.4}
 Let $X$ be a smooth complex projective surface such that
 $\overline{\NE}(X) = \bigl\{ a_1[C_1] + a_2[C_2] +\cdots+a_n[C_n] \mid a_1,\cdots,a_n \in \mathbb{R}_{\geq 0}\bigr\}$ for some irreducible curves
 $C_1,C_2,\cdots,C_n$ of $X$. 
 Let $E = L_1\oplus L_2\oplus\cdots\oplus L_r$ be a completely decomposable vector bundle of rank $r$ on $X$. 
 Let $\xi$  be the numerical class of $\mathcal{O}_{\mathbb{P}(E)}(1)$.
 Then the nef cone
 
  $\Nef\bigl(\mathbb{P}(E)\bigr)$
  
  $=\Bigl\{y_0\xi+\pi^*\gamma \mid \gamma \in N^1(X)_{\mathbb{R}}, y_0\geq0,   y_0\min\{C_i\cdot L_j\mid 1\leq j\leq r\} + (\gamma\cdot C_i) \geq 0$ for all $1\leq i \leq n\Bigr\}$.
 
\begin{proof}
 Let $C$  be an irreducible curve in $X$ numerically equivalent to $\sum\limits_{i=1}^n x_iC_i$ for $x_i \in \mathbb{R}_{\geq 0}$ for all 
 $i$ and let $\eta_c:\tilde{C}\longrightarrow C$ be its normalization.
 Then, by Lemma 2.1 in \cite{Fu}
 \vspace{2mm}
 
  $\Nef\bigl(\mathbb{P}(\eta_c^*(E\vert_C))\bigr) = \Bigl\{ a\bigl(\xi_{\tilde{c}}-\mu_{\min}(\eta_c^*(E\vert_C))f_{\tilde{c}}\bigr) + bf_{\tilde{c}} \mid a,b \in \mathbb{R}_{\geq 0}\Bigr\}$,
 \vspace{2mm}
 
  $\Nef\bigl(\mathbb{P}(\eta_{c_i}^*(E\vert_{C_{i}}))\bigr) =  \Bigl\{ a\bigl(\xi_{\tilde{c_i}}-\mu_{\min}(\eta_{c_i}^*(E\vert_{C_i}))f_{\tilde{c_i}}\bigr) + bf_{\tilde{c_i}} \mid a,b \in \mathbb{R}_{\geq 0}\Bigr\}$
  for each $i\in\{1,2,\cdots,n\}$.
  
  For any curve $C\subset X$, we define
  \begin{center}
  $l_{c} = \deg\bigl(\eta_{c}^*(E\vert_{C})\bigr) - \min\Bigl\{C\cdot L_j\mid1\leq j\leq r\Bigr\}=\deg\bigl(\eta_{c}^*(E\vert_{C})\bigr) -\mu_{\min}\bigl(\eta_{c}^*(E\vert_C)\bigr)$  
  \end{center}
  
  Applying duality, we also have
  \vspace{2mm}
  
  $\overline{\NE}\bigl(\mathbb{P}(\eta_c^*(E\vert_C))\bigr) = \Bigl\{a\bigl(\xi_{\tilde{c}}^{r-2}f_{\tilde{c}}\bigr) + b\bigl(\xi_{\tilde{c}}^{r-1}-l_c\xi_{\tilde{c}}^{r-2}f_{\tilde{c}}\bigr)\mid a,b \in \mathbb{R}_{\geq 0}\Bigr\}$,
  \vspace{2mm}
  
  $\overline{\NE}\bigl(\mathbb{P}(\eta_{c_i}^*(E\vert_{C_i}))\bigr) = \Bigl\{a\bigl(\xi_{\tilde{c_i}}^{r-2}f_{\tilde{c_i}}\bigr) + b\bigl(\xi_{\tilde{c_i}}^{r-1}-l_{c_i}\xi_{\tilde{c_i}}^{r-2}f_{\tilde{c_i}}\bigr)\mid a,b \in \mathbb{R}_{\geq 0}\Bigr\}$ for each $i\in\{1,2,\cdots,n\}$.
  \vspace{4mm}
  
  Observe that 
  \begin{center}
  $ \sum_{i=1}^n x_i\min\{C_i\cdot L_j\mid 1\leq j\leq r\}\leq \min\{\sum_{i=1}^n x_i(C_i\cdot L_j)\mid 1\leq j\leq r\}=\mu_{\min}(\eta_{c}^*(E\vert_C))$,
  \end{center}

  and $\deg(\eta_c^*(E\vert_C))=\sum_{i=1}^n x_i\deg(\eta_{c_i}^*(E\vert_{C_i})) $. Therefore
$$l_c\leq \sum_{i=1}^n x_il_{c_i}.$$

  Recall from Theorem \ref{thm1.1} that
  \vspace{2mm}
  
  $\overline{\NE}\bigl(\mathbb{P}(E)\bigr)$
  $=\overline{\sum\limits_{C\in \Gamma}\Bigl[\mathbb{R}_{\geq 0}(\xi^{r-2}F) + \mathbb{R}_{\geq 0}\bigl\{ \xi^{r-1}\pi^*[C] - \bigl(\deg(\eta_c^*(E\vert_C)) - \mu_{\min}(\eta_{c}^*(E\vert_C))\bigr) \xi^{r-2}F\bigr\}\Bigr]}$
  \vspace{3mm}
  
   $=\overline{\sum\limits_{C\in \Gamma}\Bigl[\mathbb{R}_{\geq 0}(\xi^{r-2}F) + \mathbb{R}_{\geq 0}\bigl\{ \xi^{r-1}\pi^*[C] - l_{c}\xi^{r-2}F\bigr\}\Bigr]}$
   \vspace{3mm}
  
  $= \overline{\sum\limits_{C \equiv \sum\limits_{i}x_iC_i}\Bigl[\mathbb{R}_{\geq 0}(\xi^{r-2}F) +\mathbb{R}_{\geq 0}\{\sum\limits_{i=1}^nx_i(\xi^{r-1}\pi^*[C_i] - l_{c_i}\xi^{r-2}F)+(\sum_{i=1}^n x_il_{c_i}-l_c)\xi^{r-2}F\bigl\}\Bigr]}$.
  \vspace{3mm}
  
  This shows that $\overline{\NE}\bigl(\mathbb{P}(E)\bigr)$ is generated by 
  \begin{center}
  $\Bigl\{ \xi^{r-2}F,\bigl(\xi^{r-1}\pi^*[C_i] - l_{c_i}\xi^{r-2}F\bigr)\mid 1\leq i\leq r\Bigr\}$.
  \end{center}
  
  Therefore
  \begin{center}
  $\Nef\bigl(\mathbb{P}(E)\bigr)=\Bigl\{y_0\xi+\pi^*\gamma \mid \gamma \in N^1(X)_{\mathbb{R}}, y_0\geq0,   y_0\mu_{\min}\bigl(\eta_{c_i}^*(E\vert_{C_i})\bigr) + (\gamma\cdot C_i) \geq 0$ $\forall 1\leq i \leq n\Bigr\}.$
 \end{center}
 
  But $\mu_{\min}(\eta_{c_i}^*(E\vert_{C_{i}}))=\min\{C_i\cdot L_j\mid 1\leq j\leq r\}$ for each $i\in\{1,2,\cdots,n\}$. Hence the result follows.
\end{proof}
\end{corl}
A particular case of Corollary \ref{corl3.4} is the following.
\begin{corl}\label{cor3.5}
Let $X$ be a smooth complex projective surface with Picard number 1 and $L_X$ be an ample generator of the real N\'{e}ron Severi group $N^1(X)_{\mathbb{R}}$. 
 Let $E = M_1\oplus M_2\oplus\cdots\oplus M_r$ be a completely decomposable vector bundle of rank $r\geq2$ on $X$ 
 such that $M_i \equiv a_iL_X \in N^1(X)_{\mathbb{R}}$ for each $i\in \{ 1,2,\cdots,r\}$ and $a_1\leq a_2\leq a_3 \leq\cdots\leq a_r$. 
  Then
 \begin{center}
   $\Nef\bigl(\mathbb{P}(E)\bigr) =\bigl\{ y_0\xi + y_1\pi^*L_X \mid y_0\geq 0, y_0a_1 + y_1\geq 0\bigr\}.$
  \end{center}
 In particular, if $X = \mathbb{P}^2$, and $E = \mathcal{O}_{\mathbb{P}^2}(a_1)\oplus \mathcal{O}_{\mathbb{P}^2}(a_2) \oplus\cdots\oplus \mathcal{O}_{\mathbb{P}^2}(a_r)$
 a completely decomposable bundle on $\mathbb{P}^2$ with $a_1\leq a_2\leq a_3 \leq\cdots\leq a_r$, then 
 \begin{align*}
 \Nef\bigl(\mathbb{P}(E)\bigr) = \bigl\{ y_0\xi + y_1\pi^*H \mid y_0\geq 0, y_0a_1 + y_1\geq 0\bigr\},
 \end{align*} 
 where $H$ is the numerical equivalence class of $\mathcal{O}_{\mathbb{P}^2}(1)$ in $N^1(\mathbb{P}^2)_{\mathbb{R}}$.
 \begin{proof}
 Let $m$ be the least positive integer such that $H^0(mL_X)\neq 0$ and $C_0 \in \vert mL_X \vert$ be an irreducible curve. Then we have $$\overline{\NE}(X) \subseteq \mathbb{R}_{\geq 0}L_X = \mathbb{R}_{\geq 0}\frac{1}{m}[C_0] \subseteq \Nef(X).$$ 
 
 Hence $\overline{\NE}(X) = \mathbb{R}_{\geq 0}[C_0]$.
 Now $\min\{C_0\cdot L_j \mid 1\leq j \leq r\} = \min\{ ma_jL_X^2 \mid 1\leq j \leq r\} = ma_1L_X^2$
  
  Thus by the previous Corollary \ref{corl3.4} we have
  \begin{center}
   $\Nef\bigl(\mathbb{P}(E)\bigr) = \bigl\{y_0\xi + y_1\pi^*L_X \mid y_0 \geq 0, y_0\min\{C_0\cdot L_j \mid 1\leq j \leq r\}+ (y_1L_X\cdot C_0)\geq 0 \bigr\}$
   
   $= \bigl\{ y_0\xi + y_1\pi^*L_X \mid y_0\geq 0, y_0a_1+y_1 \geq 0\bigr\}.$
  \end{center}
\end{proof}
 \end{corl}
 \begin{xrem}\label{xrem2}
  \rm We known that $\mathbb{P}^2$ is a toric variety and any line bundle is equivariant. Then the second part of Corollary \ref{cor3.5}, $\mathbb{P}(E)$ is toric variety where $E = \mathcal{O}_{\mathbb{P}^2}(a_1)\oplus \mathcal{O}_{\mathbb{P}^2}(a_2) \oplus\cdots\oplus \mathcal{O}_{\mathbb{P}^2}(a_r)$. For more detailed description of the fan structure of $\mathbb{P}(E)$ check \cite{O}, page 53. If $\Delta$ is the fan of $\mathbb{P}(E)$, then it is clear from the construction that the support of $\Delta$ has full dimension. The mori cone of curves of $\mathbb{P}(E)$ is (see \cite{CLS}, Theorem 6.3.20)
  $$\overline{\text{NE}}(\mathbb{P}(E))=\sum_{\tau \text{ a wall of } \Delta }\mathbb{R}_{\geq 0} [V(\tau)]$$
  Hence the dual of the mori cone gives the Nef cone of $\mathbb{P}(E)$.
 \end{xrem}

\begin{exm}\label{exm4.11}
\rm Let $ \phi_x: X = \tilde{\mathbb{P}}^2_x = \Bl_x\mathbb{P}^2\longrightarrow \mathbb{P}^2$ be the blow-up of $\mathbb{P}^2$ at a closed point
$x\in \mathbb{P}^2$ with exceptional divisor $\phi_x^{-1} (x) = E_x$.
Then, $\overline{\NE}(X) = \bigl\{ a[C_1] + b[C_2] \mid a,b \in \mathbb{R}_{\geq 0}\bigr\}$, where $[C_1] = \bigl[\phi_x^*\bigl(\mathcal{O}_{\mathbb{P}^2}(1)\bigr) - E_x\bigr]$, and
$[C_2] = [E_x]$. Then, $b_{11} = C_1 \cdot C_1 = 0$ , $b_{12} = b_{21} = C_1\cdot C_2 = 1$ , $b_{22} = C_2 \cdot C_2 = -1 $.

Let us consider the rank 2
bundle $ E = \phi_x^*\bigl(\mathcal{O}_{\mathbb{P}^2}(1)\bigr) \oplus \mathcal{O}_X(E_x)$ on $X$. Then, $\mathbb{P}(E)$ has Picard number 3. Fix the notations $L_1 = \phi_x^*(\mathcal{O}_{\mathbb{P}^2}(1))$ ,
$L_2 = \mathcal{O}_X(E_x)$ , and $a_{ij} = C_i\cdot L_j$ for $1 \leq i \leq 2$, $1\leq j \leq 2$. 
Then, $a_{11} = 1$ , $a_{12} = 1$ , 
$a_{21} = 0$ , $a_{22} = -1$ , $\deg(E\vert_{C_1}) = 2 $ , $\deg(E\vert_{C_2}) = -1 $.\\
$\mu_{\min}\bigl(\eta_{c_1}^*(E\vert_{c_1})\bigr) = \min\{ a_{1j} \mid 1 \leq j \leq 2\} = 1$ , $\mu_{\min}\bigl(\eta_{c_2}^*(E\vert_{c_2})\bigr) = \min\{ a_{2j} \mid 1 \leq j \leq 2\} = -1$ ,
 
Therefore, $y_0\xi+\sum\limits_{i=1}^2 y_i(\pi^*[C_i]) \in N^1\bigl(\mathbb{P}(E)\bigr)_{\mathbb{R}}$ is in $\Nef\bigl(\mathbb{P}(E)\bigr)$ if and only if 
 $y_0 \geq 0$ , \\
 $ y_0\mu_{\min}\bigl(\eta_{c_1}^*(E\vert_{c_1})\bigr)+\sum\limits_{i=1}^2 y_ib_{i1}  = y_0 + y_2 \geq 0$ and
 $ y_0\mu_{\min}\bigl(\eta_{c_2}^*(E\vert_{c_2})\bigr)+\sum\limits_{i=1}^2 y_ib_{i2} = y_1 - y_0 - y_2 \geq 0$, i.e.
 $$\Nef\bigl(\mathbb{P}(E)\bigr)=\Bigl\{y_0\xi+\sum\limits_{i=1}^2 y_i(\pi^*[C_i])\mid y_0\geq 0, y_0 + y_2 \geq 0 \text{ and } y_1 - y_0 - y_2 \geq 0\Bigr\}.$$
 
 Also note that $X$ is a toric variety. The fan structure of $X$ is given in \cite{CLS}, Definition 3.3.17. Similarly, as it is mentioned in Remark \ref{xrem2}, one can give the toric description of the Nef cone of $X$.
\end{exm}
\begin{exm}\label{exm4.12}
\rm Let 
 $\rho : X =\mathbb{P}(W) \longrightarrow C$ be a ruled surface over a smooth curve $C$, defined by an
 unstable normalized rank 2 bundle $W = \mathcal{O}_C \oplus M$ for some line bundle $M$ on $C$ with $\deg(M) = l < 0$.
 Then, $\overline{\NE}(X) = \{ a\zeta + bf \mid a,b \in \mathbb{R}_{\geq 0}\}$, where 
$\bigl[\mathcal{O}_{\mathbb{P}(W)}(1)\bigr] = \zeta \in N^1(X)_{\mathbb{R}}$ and $f$ is the numerical class of a fibre of $\rho$. 

Let $\rho_x : \tilde{X}_x  = \Bl_x(X) \longrightarrow X$ be the blow up of $X$ at a closed point $x$ in the section $\sigma$ such that $\mathcal{O}_X(\sigma)\cong \mathcal{O}_{\mathbb{P}(W)}(1)$, and $E_x$ be the
exceptional divisor. Then,
$\overline{\NE}(\tilde{X}_x) =\bigl\{ a[C_1] + b[C_2] + c[C_3] \mid a,b,c \in \mathbb{R}_{\geq 0}\bigr\}$, where
$[C_1]= \rho^*_xf - E_x$, $[C_2] = \rho^*_x\zeta - E_x$, $ [C_3] = E_x$ (see Example \ref{3.3}). The intersection products,

$b_{11} := C_1\cdot C_1 = (\rho^*_xf - E_x)\cdot(\rho^*_xf - E_x) = E_x^2 = -1$ ; $b_{33} := C_3\cdot C_3 = E_x^2 = -1$;

$b_{22} := C_2\cdot C_2 =(\rho^*_x\zeta - E_x)\cdot(\rho^*_x\zeta - E_x) = \deg(W) - 1 = l -1$ ;

$b_{12} = b_{21} := C_1\cdot C_2=(\rho^*_xf - E_x)\cdot(\rho^*_x\zeta - E_x) = \zeta\cdot f+ E_x^2 = 0$ ; 

$b_{13} = b_{31} := C_1\cdot C_3 = (\rho^*_xf - E_x)\cdot E_x = 1$ ; $b_{23} = b_{32} := C_2\cdot C_3 = (\rho^*_x\zeta - E_x)\cdot E_x=1$.

Let us consider the rank 2 bundle $E = \mathcal{O}_{\tilde{X}_x}\oplus\mathcal{O}_{\tilde{X}_x}(nE_x)$ on $\tilde{X}_x$, for some $n \in \mathbb{Z}$ with $ n\geq 0 $. Then, $\mathbb{P}(E)$ has Picard number 4 in this case.
Fix the notations $L_1 := \mathcal{O}_{\tilde{X}_x}$, $L_2 := \mathcal{O}_{\tilde{X}_x}(nE_x)$, and $a_{ij} = C_i\cdot L_j$ 
for $1\leq i \leq 3$, $1\leq j\leq 2$. Then, $a_{11}=0$ , $a_{12} = n$ , $a_{21} = 0$, $a_{22} = n$ , $a_{31} = 0$ , $a_{32} = -n$. Hence
 $\mu_{\min}\bigl(\eta_{c_1}^*(E\vert_{C_1})\bigr) = \min\{ a_{1j} \mid 1 \leq j \leq 2\} =  0$ ; 
 $ \mu_{\min}\bigl(\eta_{c_2}^*(E\vert_{C_2})\bigr) = \min\{ a_{2j} \mid 1 \leq j \leq 2\} =  0 $ ;
 $\mu_{\min}\bigl(\eta_{c_3}^*(E\vert_{C_3})\bigr) = \min\{ a_{3j} \mid 1 \leq j \leq 2\} = - n $.
 
Also, $\deg(E\vert_{C_1}) = (\mathcal{O}_{\tilde{X}_x} \oplus nE_x)\cdot(\rho^*_xf - E_x) = n$ ; 

Similarly, $\deg(E\vert_{C_2}) =n$ ; $\deg(E\vert_{C_3}) = -n$. 
 
 Therefore, $y_0\xi+\sum\limits_{i=1}^3 y_i(\pi^*[C_i]) \in N^1\bigl(\mathbb{P}(E)\bigr)_{\mathbb{R}}$ is in $\Nef\bigl(\mathbb{P}(E)\bigr)$  if and only if
 $y_0 \geq 0$ , \\
 $ y_0\mu_{\min}\bigl(\eta_{c_1}^*(E\vert_{C_1})\bigr)+\sum\limits_{i=1}^3 y_ib_{i1} = - y_1 + y_3 \geq 0$,\\
 $ y_0\mu_{\min}\bigl(\eta_{c_2}^*(E\vert_{C_2})\bigr)+\sum\limits_{i=1}^3 y_ib_{i2} =(l-1)y_2 +y_3 \geq 0$ and \\
 $ y_0\mu_{\min}\bigl(\eta_{c_3}^*(E\vert_{C_3})\bigr)+\sum\limits_{i=1}^3 y_ib_{i3} = -ny_0 + y_1 + y_2 -y_3 \geq 0$.
\end{exm}
\section{Seshadri Constants}
 We begin this section by the following easy observation.
 \begin{lem}
  Let $Z\subseteq X$ be a closed subvariety in a smooth complex projective variety $X$. For a vector bundle $V$  on $X$ and  for any point $x\in Z$ 
  \begin{center}
   $0<\varepsilon(V,x) \leq \varepsilon(V\vert_Z,x).$
  \end{center}
\begin{proof}
 We have the following commutative diagram.
\begin{center}
\begin{tikzcd} 
 \Bl_xZ \arrow[r, "j"] \arrow[d, "\psi_x"]
 & \Bl_xX \arrow[d,"\rho_x"]\\
 Z \arrow[r, "i" ]
 & X
\end{tikzcd}
\end{center}
where $j$ and $i$ are inclusions. Note that the exceptional divisor in $\Bl_xZ$ is the restriction of the exceptional divisor in $\Bl_xX$, as $\Bl_xZ$ is the strict transform of $Z$ in $\Bl_xX$.

Also we have the following commutative diagram.
\begin{center}
\begin{tikzcd} 
 \mathbb{P}\bigl(\psi^*_x(V\vert_Z)\bigr) = \mathbb{P}(j^*\rho_x^*V) \arrow[r, "l "] \arrow[d, "\pi' "]
 & \mathbb{P}(\rho_x^*V) \arrow[d," \pi "]\\
 \Bl_xZ  \arrow[r, "j" ]
 & \Bl_xX
\end{tikzcd}
\end{center}
In the above diagram we have the inclusion $\mathbb{P}\bigl(\psi^*_x(V\vert_Z)\bigr) \hookrightarrow \mathbb{P}(\rho_x^*V) $ and $$\mathcal{O}_{\mathbb{P}(\rho_x^*V)}(1)\vert_{\mathbb{P}(\psi^*_x(V\vert_Z))} = \mathcal{O}_{\mathbb{P}(\psi^*_x(V\vert_Z))}(1).$$
As restriction of nef divisor is nef, we conclude from the definition of Seshadri constant that
\begin{center}
   $0<\varepsilon(V,x) \leq \varepsilon(V\vert_Z,x).$
  \end{center}
\end{proof}
\end{lem}

 We recall from [\cite{Har1},Chapter 5] that a vector bundle $W$ of rank 2 on a smooth projective curve $C$ is said to be normalized if $H^0(W) \neq 0$, but $H^0(W \otimes L) = 0$ for all line bundle $L$ on $C$ with $\deg(L) < 0$. We notice that a normalized bundle $W$ is semistable if and only if $\deg(W) \geq 0$.
\begin{thm}\label{thm4.2}
Let $\rho : Y =\mathbb{P}(W) \longrightarrow C$ be a ruled surface over a smooth curve defined by a normalized rank 2 bundle $W$ such that $\mu_{\min}(W) = \deg(W)$. Let $\sigma$ be a section of $\rho$ such that $\mathcal{O}_Y(\sigma) \cong \mathcal{O}_{\mathbb{P}(W)}(1)$ and  $f$ denotes a fibre of the map $\rho$.
If $V$ is a semistable ample bundle with discriminant $\triangle(V)=0$, then the Seshadri constant at a closed point $y\in \sigma$ is given by
\begin{align*}
 \varepsilon(V,y) = \min\Bigl\{\mu(V\vert_\sigma),\mu(V\vert_f)\Bigr\}.
\end{align*}
\end{thm}
\begin{proof}
Consider the following fibre product diagram.
\begin{center}
\begin{tikzcd} 
 \mathbb{P}(E) = \mathbb{P}\bigl(\rho^*_y(V)\bigr) \arrow[r, "\tilde{\rho_y}"] \arrow[d, "\pi"]
 & \mathbb{P}(V) \arrow[d,"\pi'"]\\
 X :=  \Bl_yY \arrow[r, "\rho_y" ]
 & Y
\end{tikzcd}
\end{center}

Since $\mu_{\min}(W)=\deg(W)$, we have $\overline{\NE}(Y) = \{a[\sigma]+b[f]\mid a,b\in\mathbb{R}_{\geq 0}\}$. We denote the exceptional divisor of the map $\rho_y$ by $E_y$. We also note that
$$\overline{\NE}(X) =\bigl\{ a[C_1] + b[C_2] + c[C_3] \mid a,b,c \in \mathbb{R}_{\geq 0}\bigr\},$$ where
 $C_1\equiv(\rho^*_y\sigma - E_y)$, $C_2\equiv(\rho^*_yf - E_y)$ and $ C_3\equiv E_y$ (see Example \ref{3.3}). 
Note that all three curves $C_1$, $C_2$ and $C_3$ are smooth.
We define $E:= \rho^*_y(V)$.

Now, if $V$ is a semistable ample bundle with $\triangle(V)=0$, then the pullback bundle $E=\rho^*_y(V)$ is also semistable bundle with $\triangle(E)=0$ by Theorem 1.2 in \cite{B-B}.  Let $\tilde{\xi}_y$ be the numerical equivalence class of the tautological bundle $\mathcal{O}_{\mathbb{P}(\rho^*_yV)}(1)$, and 
 $\tilde{E}_y := \tilde{\rho_y}^{-1}(F_y)$, where $F_y$ is the class of the fibre of the map $\pi'$ over the point $y$. Therefore, applying Corollary \ref{corl3.1} we  get that for a positive real number $\lambda >0$,
  $\tilde{\xi_y} - \lambda\tilde{E_y}$ is nef if and only if 
  
  $\mu(E\vert_{C_1}) - \lambda(C_3\cdot C_1) = \bigl(\mu(V\vert_{f}) - \lambda\bigr) \geq 0,$
  $\mu(E\vert_{C_2}) - \lambda(C_3\cdot C_2) = \bigl(\mu(V\vert_{\sigma}) -\lambda\bigr) \geq 0,$
  
  and $\mu(E\vert_{C_3}) - \lambda(C_3\cdot C_3) = \lambda \geq 0$.\\
Therefore, the Seshadri constant at a closed point $y\in \sigma$ is given by
\begin{center}
$ \varepsilon(V,y) =  \sup\Bigl\{ \lambda > 0 \mid \tilde{\xi_y} - \lambda\tilde{E_y}$  is nef $\Bigr\}=\min\Bigl\{\mu(V\vert_\sigma),\mu(V\vert_f)\Bigr\}.$
\end{center}
\end{proof}
\begin{xrem}
 \rm Let $X:\mathbb{P}\bigl(\mathcal{O}_C\oplus\mathcal{L}\bigr)\longrightarrow C$ be a ruled surface defined by a normalized rank 2 bundle $W = \mathcal{O}_C\oplus\mathcal{L}$ with $\deg(\mathcal{L})<0$.
 The above result generalizes known results about Seshadri constants of ample line bundles on such ruled surfaces (see Theorem 4.1 in \cite{Gar06}).
 \end{xrem}
 
\begin{exm}
\rm Let $V = L_1\oplus L_2\oplus\cdots\oplus L_r$ is a completely decomposable ample bundle on $\mathbb{P}^2$. Consider the following commutative fibred diagram.
\begin{center}
\begin{tikzcd} 
 \mathbb{P}(E) := \mathbb{P}\bigl(\phi_x^*(V)\bigr) \arrow[r, "\tilde{\phi_x}"] \arrow[d, "\pi"]
 & \mathbb{P}(V) \arrow[d,"\pi'"]\\
 X = \Bl_x\mathbb{P}^2 \arrow[r, "\phi_x" ]
 & \mathbb{P}^2
\end{tikzcd}
\end{center}
We define $E:= \phi_x^*(V)$. Let $l$ be a line in $\mathbb{P}^2$ passing through $x$ and $\phi_x^*(l) = H_x$. We denote the exceptional divisor of the map $\phi_x$ by $E_x$.
Then we  have $\overline{\NE}(X) = \bigl\{ a[C_1] + b[C_2] \mid a,b \in \mathbb{R}_{\geq 0}\bigr\}$, where $C_1\equiv H_x-E_x$, $C_2\equiv E_x$. Note that both $C_1$ and $C_2$ are smooth curves. We denote the numerical equivalence class of the tautological bundle $\mathcal{O}_{\mathbb{P}(E)}(1)$ by $\tilde{\xi_{x}}$, and define $\tilde{E_x}:= \pi^*(E_x)$.

Using Corollary \ref{corl3.4} we get that for a positive real number $\lambda > 0$, 
 $\tilde{\xi_x} - \lambda\tilde{E_x} = \tilde{\xi_x} - \lambda\pi^*[C_2]$ is nef if and only if
 
 $\lambda \geq 0$, and $\min\limits_{1\leq i\leq r}\bigl\{L_i\cdot l\bigr\} - \lambda  =\min\limits_{1\leq i\leq r}\bigl\{\deg(L_i)\bigr\} - \lambda = \mu_{\min}(V) - \lambda \geq 0$.
 
 Hence, the Seshadri constant of $V = L_1\oplus L_2\oplus\cdots\oplus L_r$ at $x\in\mathbb{P}^2$  is
\begin{center}
 $ \varepsilon(V,x) = \sup\Bigl\{ \lambda > 0 \mid \tilde{\xi_x} - \lambda\tilde{E_x}$  is nef $\Bigr\} = \mu_{\min}(V) =\min\limits_{1\leq i\leq r}\bigl\{\deg(L_i)\bigr\}.$
\end{center}
 A completely decomposable ample vector bundle on $\mathbb{P}^2$  is a special case of a torus equivariant vector bundle over a toric variety. The Seshadri constant of a nef  toric vector bundle on a toric variety is calculated in [\cite{HMP}, Proposition 3.2] which also agrees with our result.  
\end{exm}

 \subsection*{\rm ACKNOWLEDGEMENT}
 The authors would like to thank Prof. Indranil Biswas, Prof. D S Nagaraj and Dr. Krishna Hanumanthu for several useful discussions. The authors would also like to thank
the referee  for his/her valuable comments and useful suggestions towards the overall improvement of the content of this article. This work is supported financially by a fellowship from TIFR, Mumbai under DAE, Government of India.


\begin{thebibliography}{***}
\normalsize
\baselineskip=18pt

\bibitem{B-D-H-K-K-S-S}{Thomas Bauer, Sandra Di Rocco, Brian Harbourne, Micha\l{}  Kapustka,
Andreas Knutsen, Wioletta Syzdek, and Tomasz Szemberg :}
\emph{A primer on Seshadri constants,}
Contemp. Math. \bf 496\rm, Amer. Math. Soc., Providence, RI (2009).

\bibitem{B-B}{Indranil Biswas, Ugo Bruzzo :}
\emph{On Semistable Principal Bundles over a Complex Projective Manifold,}
International Mathematics Research Notices. Article ID rnn035 (2008).

\bibitem{BHN20}{Indranil Biswas, Krishna Hanumanthu and D.S. Nagaraj : }
\emph{Positivity of vector bundles on homogeneous varieties,}
International Journal of Mathematics. Vol. \bf 31\rm, No. \bf 12 \rm (2020), 2050097.

\bibitem{CLS} { David A. Cox, John B. Little, and Henry K. Schenck :} 
\emph{Toric varieties}, Volume \bf 124 \rm, Graduate Studies in
Mathematics. American Mathematical Society, Providence, RI. (2011).

\bibitem{Fu} {Mihai Fulger :}
\emph{The cones of effective cycles on projective bundles over curves,}
Math.Z. (2011), 449-459.

\bibitem{Gar06} {Luis Fuentes Garci\'{a} :} 
\emph{Seshadri constants on ruled surfaces: the rational and elliptic cases,}
Manuscripta Math. \bf 119\rm (2006), 483-505.

\bibitem{Hac} {Christopher Hacon :}
\emph{Remarks on Seshadri constants of vector bundles,}
Annales de l'institut Fourier. tome \bf 50\rm, no \bf 3 \rm (2000), 767-780.

\bibitem{Har1} {Robin Hartshorne :}
\emph{Algebraic Geometry,}
Graduate Text in Mathematics, Springer (1977).

\bibitem{HMP}{Milena Hering, Mircea Mustat\'{a} and Sam Payne :}
\emph{Positivity properties of toric vector bundles,}
 Ann. Inst. Fourier (Grenoble) \bf 60\rm, no. \bf 2 \rm (2010), 607-640.
 
\bibitem{K-M-R} {Rupam Karmakar, Snehajit Misra, Nabanita Ray :}
\emph{Nef and Pseudoeffective cones of product of projective bundles over a curve,}
Bull. Sci. Math., \bf 151 \rm (2019), 1-12.

\bibitem{K-M} {Rupam Karmakar and Snehajit Misra :}
\emph{Nef cone and Seshadri constants of product of projective bundles over  curves,}
Journal of the Ramanujan Mathematical Society, \bf 35(4) \rm (2020), 317-325.

\bibitem{L1}{ Robert Lazarsfeld :}
\emph{Positivity in Algebraic Geometry,}
Volume I, Springer (2007).

\bibitem{M} { Yoichi Miyaoka :}
\emph{The Chern classes and Kodaira Dimension of a Minimal Variety,}
Advanced Studies in Pure Mathematics \bf 10 \rm, Algebraic Geometry, Sendai, 1985. (1987) 449-476.

\bibitem{M-O-H}{ Roberto Mu\~{n}oz, Gianluca Occhetta, and Luis E. Sol\'{a} Conde :}
\emph{On rank 2 vector bundles on Fano manifolds,}
Kyoto Journal of Mathematics, Vol. \bf 54 \rm, No. \bf 1 \rm (2014), 167-197.

\bibitem{O}{ Tadao Oda :}
\emph{Convex bodies and algebraic geometry,}
Springer-Verlag Berlin Heidelberg, (1988).

\end{thebibliography}
\end{document}